\input ./style/arxiv-vmsta.cfg
\documentclass[numbers,compress,v1.0.1]{vmsta}

\usepackage{cleveref}

\volume{3}
\issue{4}
\pubyear{2016}
\firstpage{287}
\lastpage{302}
\doi{10.15559/16-VMSTA67}

\startlocaldefs

\allowdisplaybreaks

\DeclareMathOperator{\pr}{\mathsf P}
\DeclareMathOperator{\M}{\mathsf E}
\DeclareMathOperator{\cov}{\mathbf{cov}}
\DeclareMathOperator{\D}{\mathsf D}

\newcommand{\prob}{\stackrel{\text{\rm P}}{\rightarrow}}
\newcommand{\probone}{\stackrel{\text{\rm P1}}{\rightarrow}}
\newcommand{\convdistr}{\stackrel{\text{\rm d}}{\rightarrow}}
\newcommand{\distr}{\stackrel{\text{\rm d}}{=}}

\renewcommand{\theequation}{\arabic{section}.\arabic{equation}}
\numberwithin{equation}{section}
\crefname{equation}{equation}{}

\newtheorem{thm}{Theorem}
\newtheorem{lemma}[thm]{Lemma}

\theoremstyle{definition}
\newtheorem{defin}[thm]{Definition}
\newtheorem{rem}[thm]{Remark}
\renewcommand{\theenumi}{\roman{enumi}}
\renewcommand{\labelenumi}{(\roman{enumi}).}

\endlocaldefs

\begin{document}
\begin{frontmatter}

\title{Goodness-of-fit test in a multivariate errors-in-variables model
$\boldsymbol{AX = B}$}
%



\author{\inits{A.}\fnm{Alexander}\snm{Kukush}\corref{cor1}}\email{alexander\_kukush@univ.kiev.ua}
\cortext[cor1]{Corresponding author.}
\author{\inits{Ya.}\fnm{Yaroslav}\snm{Tsaregorodtsev}}\email{777Tsar777@mail.ru}
\address{Taras Shevchenko National University of Kyiv, \xch{Kyiv}{Kuiv}, Ukraine}

\markboth{A. Kukush, Ya. Tsaregorodtsev}{Goodness-of-fit test in a
multivariate errors-in-variables model $AX = B$}

%


\begin{abstract}
We consider a multivariable functional errors-in-variables model \hbox
{$AX \approx B$}, where the data matrices $A$ and $B$ are observed with
errors, and a matrix parameter $X$ is to be estimated. A
goodness-of-fit test is constructed based on the total least squares
estimator. The proposed test is asymptotically chi-squared under null
hypothesis. The power of the test under local alternatives is discussed.
\end{abstract}


\received{10 November 2016}
%
\revised{3 December 2016}
%
\accepted{4 December 2016}
\publishedonline{20 December 2016}
\end{frontmatter}

\section{Introduction}\label{s:1}

We study an overdetermined system of linear equations $AX \approx B$,
which often occurs in the problems of dynamical system identification
\cite{huva}. If matrices $A$ and $B$ are observed with additive
uncorrelated errors of equal size, then the total least squares (TLS)
method is used to solve the system \cite{huva}.

In papers \cite{gl,kuhu,skl}, under various
conditions, the consistency of the TLS estimator $\hat{X}$ is proven
as the number $m$ of rows in the matrix $A$ is increasing, assuming
that the true value $A^0$ of the input matrix is nonrandom. The
asymptotic normality of the estimator is studied in \cite{gl} and \cite{kutsa}.

The model $AX \approx B$ with random measurement errors corresponds to
the vector linear errors-in-variables model (EIVM). In \cite{chku}, a
goodness-of-fit test is constructed for a polynomial EIVM with
nonrandom latent variable (i.e., in the \textit{functional} case); the
test can be also used in the \textit{structural} case, where the latent
variable is random with unknown probability distribution. A more
powerful test in the polynomial EIVM is elaborated in \cite{hallma}.

In the paper \cite{kupa}, a goodness-of-fit test is constructed for the
functional model $AX \approx B$, assuming that the error matrices
$\tilde{A}$ and $\tilde{B}$ are independent and the covariance
structure of $\tilde{A}$ is known. In the present paper, we construct a
goodness-of-fit test in a more common situation, where the total
covariance structure of the matrices $\tilde{A}$ and $\tilde{B}$ is
known up to a scalar factor. A test statistic is based on the TLS
estimator $\hat{X}$. Under the null hypothesis, the asymptotic behavior
of the test statistic is studied based on results of \cite{kutsa} and,
under local alternatives, based on \cite{skl}.

The present paper is organized as follows. In Section~\ref{s:2}, we
describe the observation model, introduce the TLS estimator, and
formulate known results on the strong consistency and asymptotic
normality of the estimator. In the next section, we construct the
goodness-of-fit test and show that the proposed test statistic has an
asymptotic chi-squared distribution with the corresponding number of
degrees of freedom. The power of the test with respect to the local
alternatives is studied in Section~\ref{s:4}, and Section~\ref{s:5}
concludes. The proofs are given in Appendix.

We use the following notation: $\|C\| = \sqrt{\sum_{i,j}c_{ij}^2}$ is
the Frobenius norm of a matrix $C = (c_{ij})$, and $\mathrm{I}_p$ is
the unit matrix of size $p$. The symbol $\M$ denotes the expectation
and acts as an operator on the total product of quantities, and $\cov$
means the covariance matrix of a random vector. The upper index $\top$
denotes transposition. In the paper, all the vectors are column ones.
The bar means averaging over\break $i=1, \ldots, m$, for example, $\bar
{a}:=m^{-1}\sum_{i=1}^{m}a_i$, $\overline{ab^{\top
}}:=m^{-1}\sum_{i=1}^{m}a_ib_i^{\top}$. Convergence with
probability one, in probability, and in distribution are denoted as
$\probone$, $\prob$, and $\convdistr$, respectively. A sequence of
random matrices that converges to zero in probability is denoted as
$o_p(1)$, and a sequence of stochastically bounded random matrices is
denoted as $O_p(1)$. The notation $\varepsilon\distr\varepsilon_1$
means that random variables $\varepsilon$ and $\varepsilon_1$ have the
same probability distribution. Positive constants that do not depend on
the sample size $m$ are denoted as $\mathit{const}$, so that equalities like
$2\cdot \mathit{const} = \mathit{const}$ are possible.

\section{Observation model and total least squares estimator}\label{s:2}
\subsection{The TLS problem}\label{s:2.1}

Consider the observation model
\begin{equation}
\label{OLSy} A^0X^0 = B^0, \qquad A =
A^0 + \tilde{A}, \qquad B = B^0 + \tilde{B},
\end{equation}
where $A^0\in\mathbb{R}^{m\times n}$, $X^0\in\mathbb{R}^{n\times d}$,
and $B^0\in\mathbb{R}^{m\times d}$. The matrices $A$ and $B$ contain
the data, $A^0$ and $B^0$ are unknown nonrandom matrices, and $\tilde
{A}$, $\tilde{B}$ are the matrices of random errors.

We can rewrite model \eqref{OLSy} in an implicit way. Introduce three
matrices of size $m\times(n+d)$:
\begin{equation}
\label{EM} C^0:= \bigl[A^0 \ \ B^0\bigr],
\qquad \tilde{C}:= [\tilde{A} \ \ \tilde{B}], \qquad C:= [A \ \ B].
\end{equation}
Then
\[
C = C^0 + \tilde{C}, \qquad C^0\cdot %
\begin{bmatrix}
X^0\\
-\mathrm{I}_d
\end{bmatrix}
=0.
\]

Let $A^{\top} = [a_1\dots a_m]$, $B^{\top} = [b_1\dots b_m]$, and we
use similar notation for the rows of the matrices $C$, $A^0$, $B^0$,
$\tilde{A}$, $\tilde{B}$, and $\tilde{C}$. Rewrite model \eqref{OLSy}
as a~multivariate linear one:
\begin{gather}
X^{0\top}a_i^0 = b_i^0,
\label{MLM1}
\\
b_i = b_i^0 + \tilde{b}_i,
\qquad a_i = a_i^0 + \tilde{a}_i;\quad
i=1, \ldots, m.\label{MLM2}
\end{gather}

Throughout the paper, the following assumption holds about the errors\break
$\tilde{c_i}=[\tilde{a_i}^\top\tilde{b}_i^\top]^\top$:
\renewcommand{\theenumi}{\roman{enumi}}
\renewcommand{\labelenumi}{(\roman{enumi})}
\begin{enumerate}
\item\label{i}
The vectors $\tilde{c}_i$, $i\geq1$, are i.i.d. with zero mean, and, moreover,
\begin{equation}
\label{CE} \cov(\tilde{c}_1) = \sigma^2
\mathrm{I}_{n+d},
\end{equation}
with unknown $\sigma>0$.
\end{enumerate}

Thus, the total error covariance structure is assumed to be known up to
a~scalar factor $\sigma^2$, and the errors are uncorrelated with equal
variances.

For model \eqref{OLSy}, the TLS problem lies in searching such
disturbances $\Delta\hat{A}$ and $\Delta\hat{B}$ that minimize the sum
of squared corrections
\begin{equation}
\label{Min1} \min_{(X\in\mathbb{R}^{n\times d},\Delta A,\Delta B)}\bigl(\|\Delta A\|^2 + \|
\Delta B\|^2\bigr),
\end{equation}
provided that
\begin{equation}
\label{Min2} (A - \Delta A)X = B - \Delta B.
\end{equation}
\subsection{The TLS estimator and its consistency}\label{s:2.2}

It can happen that for certain random realization, the optimization
problem \eqref{Min1}--\eqref{Min2} has no solution. In the latter case,
we set $\hat{X} = \infty$.

\begin{defin}
The TLS estimator $\hat{X}$ of the matrix parameter $X^0$ in the model
\eqref{OLSy} is a Borel-measurable function of the observed matrices
$A$ and $B$ such that its values lie in $\mathbb{R}^{n\times d}\cup\{
\infty\}$ and it provides a solution to problem \eqref{Min1}--\eqref
{Min2} in case there exists a solution, and $\hat{X} = \infty$ otherwise.
\end{defin}

We need the following conditions to provide the consistency of the estimator:
\begin{enumerate}\addtocounter{enumi}{1}
\item\label{ii}
$\M\|\tilde{c}_1\|^4 < \infty$.
\item\label{iii}
$\frac{1}{m} A^{0\top}A^0 \to V_A$ as $m\to\infty$, where $V_A$ is a
nonsingular matrix.
\end{enumerate}

The next result on the strong consistency of the estimator follows, for
example, from Theorem~4.3 in \cite{skl}.
\begin{thm}\label{th:2}
Assume conditions \eqref{i}--\eqref{iii}. Then, with probability one,
for all\break $m\geq m_0(\omega)$, the TLS estimator $\hat{X}$ is finite,
and, moreover, $\hat{X}\probone X^0$ as $m\to\infty$.
\end{thm}

Define the loss function $Q(X)$ as follows:
\begin{align}
q(a,b;X)&:=\bigl(a^{\top}X-b^{\top}\bigr) \bigl(
\mathrm{I}_d+X^{\top}X\bigr)^{-1}
\bigl(X^{\top
}a-b\bigr),\label{LFE}
\\
Q(X)&:=\sum_{i=1}^m q(a_i,b_i;X),
\quad X\in\mathbb{R}^{n\times d}. \label{LF}
\end{align}
It is known that the TLS estimator minimizes the loss function \eqref
{LF}; see formula (24) in \cite{kuhu}.

Introduce the following unbiased estimating function related to the
elementary loss function \eqref{LFE}:
\begin{equation}
\label{EF} s(a,b;X)=a\bigl(a^{\top}X-b^{\top}\bigr)-X\bigl(
\mathrm{I}_d+X^{\top}X\bigr)^{-1}
\bigl(X^{\top
}a-b\bigr) \bigl(a^{\top}X-b^{\top}\bigr).
\end{equation}
\begin{lemma}\label{l:3}
Assume conditions \eqref{i}--\eqref{iii}. Then, with probability one,
for all\break $m\geq m_0(\omega)$, the TLS estimator $\hat{X}$ is a solution
to the equation
\[
\sum_{i=1}^m s(a_i,b_i;X)=0,
\quad X\in\mathbb{R}^{n\times d}.
\]
\end{lemma}

In view of Theorem \ref{th:2}, the statement of Lemma \ref{l:3} follows
from Corollary~4(a) in \cite{kutsa}.
\subsection{Asymptotic normality of the estimator}\label{s:2.3}

We need further restrictions on the model. Recall that the augmented
errors~$\tilde{c}_i$ were introduced in Section~\ref{s:2.2}, and the
vectors $a_i^0$, $\tilde{b}_i$, and so on are those from model \eqref
{MLM1}--\eqref{MLM2}.
\begin{enumerate}\addtocounter{enumi}{3}
\item\label{iv}
$\M\|\tilde{c}_1\|^{4 + 2\delta} < \infty$ for some $\delta>0$;
\item\label{v}
For $\delta$ from condition \eqref{iv}, $\frac{1}{m^{1 +
\delta/2}}\sum_{i=1}^{m}\|a_i^0\|^{2 + \delta}\to0$ as $m\to\infty$.
\end{enumerate}

Denote by $\tilde{c}_1^{(p)}$ the $p$th coordinate of the vector $\tilde{c}_1$.
\begin{enumerate}\addtocounter{enumi}{5}
\item\label{vi} For all $p,q,r=1,\ldots, n+d$, we have $\M\tilde
{c}_1^{(p)}\tilde{c}_1^{(q)}\tilde{c}_1^{(r)} = 0$.
\end{enumerate}

Under assumptions \eqref{i} and \eqref{iv}, condition \eqref{vi} holds,
for example, in two cases: (a)~when the random vector $\tilde{c}_1$ is
symmetrically distributed, or (b) when the components of the vector
$\tilde{c}_1$ are independent and, moreover, for each $p=1, \ldots,
n+d$, the asymmetry coefficient of the random variable $\tilde
{c}_1^{(p)}$ equals~0.

Introduce the following random element in the space of collections of
five matrices:
\begin{equation}
\label{RM} W_i = \bigl(a_i^0
\tilde{a}_i^{\top},a_i^0
\tilde{b}_i^{\top},\tilde{a}_i\tilde
{a}_i^{\top}-\sigma^2\mathrm{I}_n,
\tilde{a}_i\tilde{b}_i^{\top},\tilde
{b}_i\tilde{b}_i^{\top}-\sigma^2
\mathrm{I}_d\bigr).
\end{equation}

The next statement on the asymptotic normality of the estimator follows
from the proof of Theorem 8(b) in \cite{kutsa}, where, instead of
condition \eqref{vi}, there was a stronger assumption that $\tilde
{c}_1$ is symmetrically distributed, but the proof of Theorem 8(b) in
\cite{kutsa} still works under the weaker condition \eqref{vi}\xch{.}{).}
\begin{thm}\label{th:4}
Assume conditions \eqref{i} and \eqref{iii}--\eqref{vi}. Then:
\renewcommand{\theenumi}{(\alph{enumi})}
\renewcommand{\labelenumi}{\rm(\alph{enumi})}
\begin{enumerate}
\item\label{4.a}\ \vspace*{-18pt}
\begin{equation}
\label{CW} \displaystyle\frac{1}{\sqrt{m}}\sum_{i=1}^m
W_i\stackrel{\text{\rm d}} {\longrightarrow} \varGamma= (
\varGamma_1,\dots,\varGamma_5)\quad \text{as }~ m\to\infty,
\end{equation}
where $\varGamma$ is a Gaussian centered random element with matrix components,
\item\label{4.b}\ \vspace*{-26pt}
\begin{align}
\label{AN} \sqrt{m}\bigl(\hat{X} - X^0\bigr)&\convdistr
V_A^{-1}\varGamma\bigl(X^0\bigr)\quad \text{as }~ m
\to\infty,\\
\varGamma(X)&:=\varGamma_1 X - \varGamma_2 +
\varGamma_3 X - \varGamma_4\notag\\
\label{AN1} &\quad-X\bigl(\mathrm{I}_d + X^{\top}X
\bigr)^{-1}\bigl(X^{\top}\varGamma_3 X -
X^{\top}\varGamma _4-\varGamma_4^{\top}X +
\varGamma_5\bigr),
\end{align}
where $V_A$ is from condition \eqref{iii}, and $\varGamma_i$ is from
condition \eqref{CW}.
\end{enumerate}
\end{thm}
\begin{rem}
Under the assumptions of Theorem \ref{th:4}, the components of random
element \eqref{RM} are uncorrelated, and therefore, the components of
the limit element $\varGamma$ are uncorrelated as well.
\end{rem}

Let $f\in\mathbb{R}^{n\times1}$. Under the conditions of Theorem \ref
{th:4}, the convergence \eqref{AN} implies that
\begin{align}
&\sqrt{m}\bigl(\hat{X} - X^0\bigr)^{\top}f \convdistr N
\bigl(0,S\bigl(X^0,f\bigr)\bigr),\label{ANF}
\\
&\quad S\bigl(X^0,f\bigr) = \M\varGamma^{\top}(X_0)V_A^{-1}ff^{\top}V_A^{-1}
\varGamma (X_0).\label{ACM}
\end{align}

Let a consistent estimator $\hat{f} = \hat{f}_m$ of the vector $f$ be
given. We want to construct a consistent estimator of matrix \eqref
{ACM}. The matrix $S(X^0,f)$ is expressed, for example, via the fourth
moments of errors $\tilde{c}_i$, and those moments cannot be
consistently estimated without additional assumptions on the error
probability distribution. Therefore, an explicit expression for the
latter matrix does not help to construct the desirable estimator.
Nevertheless, we can construct something like the sandwich estimator
\cite[pp.~368--369]{car}.

The next statement on the consistency of the nuisance parameter
estimators follows from the proof of Lemma 10 in \cite{kutsa}. Recall
that the bar means averaging over the observations; see Section~\ref{s:1}.
\begin{lemma}\label{l:6}
Assume the conditions of Theorem \emph{\ref{th:4}}. Define the estimators:
\begin{align}
\label{ES} \hat{\sigma}^2 &= \frac{1}{d}\mathrm{tr} \bigl[
\bigl(\overline{bb^{\top}}- 2\hat{X}^{\top}\overline{ab^{\top}}+
\hat{X}^{\top}\overline{aa^{\top}}\hat{X}\bigr) \bigl(
\mathrm{I}_d+\hat{X}^{\top}\hat{X}\bigr)^{-1} \bigr],\\
\label{EV} \hat{V}_A &= \overline{aa^{\top}} - \hat{
\sigma}^2\mathrm{I}_n.
\end{align}
Then
\begin{equation}
\hat{\sigma}^2\prob\sigma^2, \qquad \hat{V}_A
\prob V_A.
\end{equation}
\end{lemma}

The next asymptotic expansion of the TLS estimator is presented in \cite
{kutsa}, formulas (4.10) and (4.11).
\begin{lemma}\label{l:7}
Under the conditions of Theorem \emph{\ref{th:4}}, we have:
\begin{equation}
\label{AEX} \sqrt{m}\bigl(\hat{X} - X^0\bigr) = -
V_A^{-1}\cdot\frac{1}{\sqrt{m}}\sum
_{i=1}^{m}s\bigl(a_i,b_i;X^0
\bigr) + o_p(1).
\end{equation}
\end{lemma}

In view of Lemma \ref{l:7}, introduce the sandwich estimator $\hat
{S}(\hat{f})$ of the matrix (\ref{ACM}):
\begin{equation}
\label{EACM} \hat{S}(\hat{f}) = \frac{1}{m}\sum
_{i=1}^{m}s^{\top}(a_i,b_i;
\hat {X})~\hat{V}_A^{-1}\hat{f}\hat{f}^{\top}
\hat{V}_A^{-1}~s(a_i,b_i;\hat{X}),
\end{equation}
where the estimator $\hat{V}_A$ is given in \eqref{EV}.
\begin{thm}\label{th:8}
Let $f\in\mathbb{R}^{n\times1}$, and let $\hat{f}$ be a consistent
estimator of this vector. Under the conditions of Theorem \emph{\ref{th:4}},
the statistic $\hat{S}(\hat{f})$ is a consistent estimator of the
matrix $S(X^0,f)$, that is, $\hat{S}(\hat{f})\prob S(X^0,f)$.
\end{thm}

Appendix contains the proof of this theorem and of all further statements.
\section{Construction of goodness-of-fit test}\label{s:3}
For the observation model \eqref{MLM2}, we test the following
hypotheses concerning the response $b$ and the latent variable $a^0$:

$\textbf{H}_0$ There exists such a matrix $X\in\mathbb{R}^{n\times d}$ that
\begin{equation}
\label{H0} \M\bigl(b - X^{\top}a^0\bigr) = 0,\quad \mbox{and}
\end{equation}

$\textbf{H}_1$ For each matrix $X\in\mathbb{R}^{n\times d}$,
\begin{equation}
\label{H1} \M\bigl(b - X^{\top}a^0\bigr) \text{ is not
identically zero.}
\end{equation}

In fact, the null hypothesis means that the observation model
(1.3)--(1.4) holds. Based on observations $a_i$, $b_i$, $i=1, \ldots,
m$, we want to construct a test statistic to check this hypothesis. Let
\begin{equation}
\label{ST} T_m^0 := \frac{1}{m}\sum
_{i=1}^{m}\bigl(b_i -
\hat{X}^{\top}a_i\bigr) = \overline {b -
\hat{X}^{\top}a}.
\end{equation}
\begin{lemma}\label{l:9}
Under the conditions of Theorem \emph{\ref{th:4}},
\begin{equation}
\label{STE} \sqrt{m}T_m^0 = \frac{1}{\sqrt{m}}\sum
_{i=1}^{m}\bigl(\tilde{b}_i -
X^{0\top
}\tilde{a}_i\bigr) - \sqrt{m}\bigl(\hat{X} -
X^0\bigr)^{\top}\overline{a^0} +
o_p(1).
\end{equation}
\end{lemma}

We need the following stabilization condition on the latent variable:
\begin{enumerate}\addtocounter{enumi}{6}
\item\label{vii}
$\displaystyle\frac{1}{m}\sum_{i=1}^{m}a_i^0\to\mu_a$ as $m\to\infty$
with $\mu_a\in\mathbb{R}^{n\times1}$.
\end{enumerate}
\begin{lemma}\label{l:10}
Assume conditions \eqref{i} and \eqref{iii}--\eqref{vii}. Then
\begin{align}
&\sqrt{m}T_m^0\convdistr N(0,\varSigma_T),\notag\\
&\quad\label{ACT} \varSigma_T = \sigma^2\bigl(1 - 2
\mu_a^{\top}V_A^{-1}\mu_a
\bigr) \bigl(\mathrm{I}_d + X^{0\top}X^0\bigr) + S
\bigl(X^0, \mu_a\bigr).
\end{align}
\end{lemma}
\begin{lemma}\label{l:11}
Assume the conditions of Lemma \emph{\ref{l:10}}. Then:
\renewcommand{\theenumi}{\alph{enumi}}
\renewcommand{\labelenumi}{\rm(\alph{enumi})}
\begin{enumerate}
\item\label{11.a}
A strong consistent estimator of the vector $\mu_a$ from condition
\eqref{vii} is given by the statistic
\[
\hat{\mu}_a:=\bar{a} = \frac{1}{m}\sum
_{i=1}^{m}a_i.
\]
\item\label{11.b}
A consistent estimator of matrix \eqref{ACT} is given by the matrix statistic
\begin{equation}
\label{EACS} \hat{\varSigma}_T := \hat{\sigma}^2\bigl(1 -
2\hat{\mu}_a^{\top}\hat {V}_A^{-1}
\hat{\mu}_a\bigr) \bigl(\mathrm{I}_d +
\hat{X}^{\top}\hat{X}\bigr) + \hat {S}(\hat{\mu}_a),
\end{equation}
where $\hat{\sigma}^2$ and $\hat{V}_A$ are presented in \eqref{ES} and
\eqref{EV}, respectively, and $\hat{S}(\hat{\mu}_a)$ is matrix \eqref
{EACM} with $\hat{f} = \hat{\mu}_a$.
\end{enumerate}
\end{lemma}

To ensure the nonsingularity of the matrix $\varSigma_T$, we impose a
final restriction on the observation model:
\begin{enumerate}\addtocounter{enumi}{7}
\item\label{viii}
There exists a finite matrix limit
\[
S_a := \lim_{m\to\infty}\frac{1}{m}\sum
_{i=1}^{m}\bigl(a_i^0 -
\mu_a\bigr) \bigl(a_i^0 - \mu_a
\bigr)^{\top},
\]
\end{enumerate}
and, moreover, the matrix $S_a$ is nonsingular.
\begin{rem}
Assume conditions \eqref{vii} and \eqref{viii}. Then
\[
\frac{1}{m}A^{0\top}A^0 = \frac{1}{m}\sum
_{i=1}^{m}a_i^0a_i^{0\top}
\to V_A = S_a + \mu_a\mu_a^{\top}
\quad \text{as } m\to\infty,
\]
and $V_A$ is nonsingular as a sum of positive definite and positive
semidefinite matrices. Thus, condition \eqref{iii} is a consequence of
assumptions \eqref{vii} and \eqref{viii}.
\end{rem}
\begin{lemma}\label{l:13}
Assume conditions \eqref{i} and \eqref{iv}--\eqref{viii}. Then:
\renewcommand{\theenumi}{\alph{enumi}}
\renewcommand{\labelenumi}{\rm(\alph{enumi})}
\begin{enumerate}
\item\label{13.a}
Matrix \eqref{ACT} is positive definite.
\item\label{13.b}
With probability tending to one as $m\to\infty$, the symmetric matrix
$\hat{\varSigma}_T$ is positive definite as well.
\end{enumerate}
\end{lemma}

For $m\geq1$ and $\omega$ from the underlying probability space $\varOmega
$ such that $\hat{\varSigma}_T$ is positive definite, we define the test statistic
\begin{equation}
\label{TST} T_m^2 = m\cdot\big\|\hat{\varSigma}_T^{-1/2}T_m^0\big\|^2.
\end{equation}

Lemmas \ref{l:10} and \ref{l:11}\eqref{11.b} imply the following
convergence of the test statistic.
\begin{thm}\label{th:14}
Assume conditions \eqref{i} and \eqref{iv}--\eqref{viii}. Then under
hypothesis $\textbf{H}_0$,\break $T_m^2\convdistr\chi_d^2$ as $m\to\infty$.
\end{thm}

Given a confidence level $\alpha$, $0<\alpha<1/2$, let $\chi_{d\alpha
}^2$ be the upper $\alpha$-quantile of the $\chi_{d}^2~$ probability
law, that is, $\pr\{\chi_d^2 > \chi_{d\alpha}^2\} = \alpha$. Based on
Theorem \ref{th:14}, we construct the following goodness-of-fit test
with the asymptotic confidence probability $1-\alpha$:
\begin{center}
\ \ \ \,\quad If $T_m^2\leq\chi_{d\alpha}^2$, then we accept the null hypothesis, \\
and if $T_m^2>\chi_{d\alpha}^2$, then we reject the null hypothesis.
\end{center}
\section{Power of the test}\label{s:4}
Consider a sequence of models
\begin{equation}
\label{LAL} \textbf{H}_{1,m}{:}\quad b_i = X^{\top}a_i^0
+ \frac{g(a_i^0)}{\sqrt{m}} + \tilde{b}_i, \qquad a_i =
a_i^0 + \tilde{a}_i, \quad i =1, \ldots, m.
\end{equation}
Here $g:\mathbb{R}^n\to\mathbb{R}^d$ is a given nonlinear perturbation
of the linear regression function.

For arbitrary function $f(a^0)$, denote the limit of averages
\[
M\bigl(f\bigl(a^0\bigr)\bigr) = \lim_{m\to\infty}
\overline{f\bigl(a^0\bigr)},
\]
provided that the limit exists and is finite.

In order to study the behavior of the test statistic under local
alternatives $\textbf{H}_{1,m}$, we impose two restrictions on the
perturbation function $g$:
\begin{enumerate}\addtocounter{enumi}{8}
\item\label{ix}
There exist $M(g(a^0))$ and $M(g(a^0)a^{0\top})$.
\item\label{x}
$\overline{\|g(a^0)\|^2} = o(m)$ as $m\to\infty$.
\end{enumerate}

Under local alternatives $\textbf{H}_{1,m}$, we ensure the weak
consistency and asymptotic normality of the TLS estimator $\hat{X}$.
\begin{lemma}\label{l:15}
Assume conditions \eqref{i} and \eqref{iv}--\eqref{x}. Under local
alternatives $\textbf{H}_{1,m}$, we have:
\renewcommand{\theenumi}{\alph{enumi}}
\renewcommand{\labelenumi}{\rm(\alph{enumi})}
\begin{enumerate}
\item\label{15.a}
$\hat{X}\prob X^0$, $\hat{\sigma}^2\prob\sigma^2$.
\item\label{15.b}
$\sqrt{m}(\hat{X} - X^0)\convdistr V_A^{-1}\varGamma(X^0) +
V_A^{-1}M(a^0g^{\top}(a^0))$ as $m\to\infty$,
\end{enumerate}
where $\varGamma(X)$ is defined in \eqref{RM}, \eqref{CW}, and \eqref{AN1}.
\end{lemma}
\begin{lemma}\label{l:16}
Assume the conditions of Lemma \emph{\ref{l:15}}. Then under local
alternatives $\textbf{H}_{1,m}$, we have:
\renewcommand{\theenumi}{\alph{enumi}}
\renewcommand{\labelenumi}{\rm(\alph{enumi})}
\begin{enumerate}
\item\label{16.a}
$\sqrt{m}T_m^0\convdistr N(C_T,\varSigma_T)$,

where $\varSigma_T$ is given by \eqref{ACT}, and
\begin{equation}
\label{PNC} C_T := M\bigl(g\bigl(a^0\bigr)\bigr) - M
\bigl(g\bigl(a^0\bigr)a^{0\top}\bigr)V_A^{-1}
\mu_a.
\end{equation}
\item\label{16.b}
The estimator $\hat{\varSigma}_T$ given in \eqref{EACS} tends in
probability to the asymptotic covariance matrix
$\varSigma_T$.\vadjust{\eject}
\end{enumerate}
\end{lemma}

Now, we define the noncentral chi-squared distribution $\chi_d^2(\tau)$
with $d$ degrees of freedom and the noncentrality parameter $\tau$.
\begin{defin}\label{d:17}
For $d\geq1$ and $\tau\geq0$, let $\chi_d^2(\tau)\distr\|N(\tau
e,\mathrm{I}_d)\|^2$, where $e\in\mathbb{R}^d$, $\|e\|=1$, or,
equivalently, $\chi_d^2(\tau)\distr(\gamma_1 + \tau)^2 +
\sum_{i=2}^{d}\gamma_i^2$, where $\{\gamma_i\}$ are i.i.d.\ standard
normal random variables.
\end{defin}

Lemma \ref{l:16} implies directly the following convergence.
\begin{thm}\label{th:18}
Assume conditions \eqref{i} and \eqref{iv}--\eqref{x}. Then under
local alternatives $\textbf{H}_{1,m}$, we have:
\begin{equation}
\label{CLA} T_m^2\convdistr\chi_d^2(
\tau), \qquad \tau:= \big\|\varSigma_T^{-1/2}C_T\big\|,
\end{equation}
where $C_T$ is given in \eqref{PNC}.
\end{thm}

Theorem \ref{th:18} makes it possible to find the asymptotic power of
the test under local alternatives $\textbf{H}_{1,m}$. It is evident
that the asymptotic power is an increasing function of $\tau= \|\varSigma
_T^{-1/2}C_T\|$. In other words, the larger $\tau$, the more powerful
the test.
\section{Conclusion}\label{s:5}

We constructed a goodness-of-fit test for a multivariate linear
errors-in-variables\break model, provided that the errors are uncorrelated
with equal (unknown) variances and vanishing third moments. The latter
moment assumption makes it possible to estimate consistently the
asymptotic covariance matrix $\varSigma_T$ of the statistic $T_m^0$ and
construct the test statistic $T_m^2$, which has the asymptotic $\chi
_d^2$ distribution under the null hypothesis. The local alternatives
$\textbf{H}_{1,m}$ are presented, under which the test statistic has
the noncentral $\chi_d^2(\tau)$ asymptotic distribution. The larger
$\tau$, the larger the asymptotic power of the test.

In future, we will try to construct, like in \cite{kupa}, a more
powerful test using within a test statistic the exponential weight
function
\[
\omega_{\lambda}(a) = e^{\lambda^{\top}a}, \quad \lambda\in\mathbb
{R}^{n\times1}.
\]
To this end, it is necessary to require the independence he terrors
$\tilde{b}_i$ and $\tilde{a}_i$ and also the existence of exponential
moments of the errors $\tilde{a}_i$. This is the price for a greater
power of the test.
\section*{Appendix}
\begin{lemma}\label{l:19}
Let $r>1$ be a fixed real number, and $\{\eta_k\}$ be an i.i.d.
sequence with zero mean and finite moment $\M|\eta_1|^r$. Assume also
that a sequence $\{d_k\}$ of real numbers satisfies
\[
\frac{1}{m^r}\sum_{k=1}^{m}|d_k|^r
\to0 \quad\text{as } m\to\infty.
\]
Then
\begin{equation}
\label{LLN} \overline{d\eta} = \frac{1}{m}\sum
_{k=1}^{m}d_k\eta_k\prob0.
\end{equation}
\end{lemma}
\begin{proof}
Without of loss generality, we may and do assume that $1<r<2$. It
suffices to check that the following three conditions from Theorem 5 in
[8, Chap.~VI] hold, which provide a criterion for the convergence \eqref{LLN}:

(a) $\displaystyle\sum_{k=1}^{m}\pr\{|d_k\eta_k|>m\}\leq
\sum_{k=1}^{m}\frac{\M|d_k\eta_k|^r}{m^r} =
\frac{\M|\eta_1|^r}{m^r}\sum_{k=1}^{m}|d_k|^r\to0$ as $m\to\infty$;

(b) $\displaystyle\frac{1}{m^2}\sum_{k=1}^{m}\D(d_k\eta_k\mathrm
{I}(|d_k\eta_k|<m))\leq
\frac{1}{m^2}\sum_{k=1}^{m}\M(d_k^2\eta_k^2\mathrm{I}(|d_k\eta_k|<m))$

\ \ \qquad$\leq\displaystyle\frac{1}{m^2}\sum_{k=1}^{m}\M|d_k\eta_k|^r\cdot
m^{2-r} =
\frac{\M|\eta_1|^r}{m^r}\sum_{k=1}^{m}|d_k|^r\to0$ as $m\to\infty$;

(c) $\varepsilon_m := \displaystyle\frac{1}{m}\sum_{k=1}^{m}\M(d_k\eta
_k\mathrm{I}(|d_k\eta_k|<m)) =
-\frac{1}{m}\sum_{k=1}^{m}\M(d_k\eta_k\mathrm{I}(|d_k\eta_k|\geq m))$,

\quad$|\varepsilon_m| \leq\displaystyle\frac{1}{m}\sum_{k=1}^{m}\M|d_k\eta
_k|^r\cdot\frac{1}{m^{r-1}} =
\frac{\M|\eta_1|^r}{m^r}\sum_{k=1}^{m}|d_k|^r\to0$ as $m\to\infty$.

By the mentioned theorem from \cite{petr} the presented bounds imply
the desired convergence.
\end{proof}

The next statement is a version of the Lyapunov CLT.
\begin{lemma}\label{l:20}
Let $\{z_i\}$ be a sequence of independent centered random vectors in
$\mathbb{R}^p$ with $\overline{\cov(z)} = \frac{1}{m}\sum_{i=1}^{m}\cov(z_i)\to S$ as $m\to\infty$. Assume also that, for some
$\delta>0$,
\begin{equation}
\label{BM} \frac{1}{m^{1 + \delta/2}}\sum_{i=1}^{m}
\M\|z_i\|^{2+\delta}\leq \mathit{const}.
\end{equation}
Then
\[
\frac{1}{\sqrt{m}}\sum_{i=1}^{m}
z_i\convdistr N(0,S).
\]
\end{lemma}
\begin{proof}[Proof of Theorem \ref{th:8}]
(a) We have:
\begin{align}
S(f) &:= \frac{1}{m}\sum_{i=1}^{m}s^{\top}
\bigl(a_i,b_i;X^0\bigr)~V_A^{-1}ff^{\top
}V_A^{-1}~s
\bigl(a_i,b_i;X^0\bigr)\notag\\
\label{Sf} &\ =\bigl(S(f) - \M S(f)\bigr) + \M S(f).
\end{align}

In the proof of Theorem 8(a) in \cite{kutsa}, the following expansion
of the estimating function is used:
\begin{align}
s\bigl(a_i,b_i;X^0\bigr) &=
W_{i1}X^0 - W_{i2} + W_{i3}X^0
- W_{i4} - X^0\bigl(\mathrm {I}_d +
X^{0\top}X^0\bigr)^{-1}
\\
&\quad\times\bigl(X^{0\top}W_{i3}X^0 -
X^{0\top}W_{i4} - W_{i4}^{\top}X^0
+ W_{i5}\bigr), \label{EFE}
\end{align}
where $W_{ij}$ are the components of the matrix collection \eqref{RM}.

We show that the term in parentheses on the right-hand side of \eqref
{Sf} tends to zero in probability. Taking into account expansion \eqref
{EFE}, we write down one of summands\vadjust{\eject} of the expression $S(f)$:
\begin{equation}
\label{zz} L_m := \frac{1}{m}\sum
_{i=1}^{m}X^{0\top}\tilde{a}_ia_i^{0\top
}Za_i^0
\tilde{a}_i^{\top}X^0, \quad Z :=
V_A^{-1}ff^{\top}V_A^{-1}.
\end{equation}
Let us explain why
\begin{equation}
\label{Lm} L_m - \M L_m \prob0.
\end{equation}
It suffices to consider the matrix
\[
\tilde{L}_m := \frac{1}{m}\sum_{i=1}^{m}
\tilde{a}_ia_i^{0\top
}Za_i^0
\tilde{a}_i^{\top}.
\]
Up to a constant, its entries contain summands of the form
\[
\frac{1}{m}\sum_{i=1}^{m}
\tilde{a}_i^{(j)} a_i^{0(p)}
a_i^{0(q)}\tilde {a}_i^{(r)}.
\]
Applying Lemma \ref{l:19} to the expression
\begin{equation}
\label{CS} \frac{1}{m}\sum_{i=1}^{m}a_i^{0(p)}
a_i^{0(q)} \bigl(\tilde{a}_i^{(j)}
\tilde{a}_i^{(r)} - \M\tilde{a}_i^{(j)}
\tilde{a}_i^{(r)} \bigr),
\end{equation}
we have $\M (\tilde{a}_i^{(j)} \tilde{a}_i^{(r)} )^2 \leq\M
\|\tilde{a}_i\|^4 <\infty$, and for $\delta$ from condition \eqref{v},
we have:
\[
\frac{1}{m^{1 + \delta/2}}\sum_{i=1}^{m}\big|a_i^{0(p)}
a_i^{0(q)}\big|^{1 +
\delta/2}\leq \frac{1}{m^{1 + \delta/2}}\sum
_{i=1}^{m}\big\|a_i^0\big\|^{2 + \delta}
\to0\quad \text{as } m\to\infty.
\]
Thus, by Lemma \ref{l:19} expression \eqref{CS} tends to zero in
probability. Then
\[
\tilde{L}_m - \M\tilde{L}_m \prob0,
\]
whence we get \eqref{Lm}.

In a similar way, other summands of $S(f)$ can be studied, and therefore,
\[
S(f) - \M S(f)\prob0.
\]
Next, we verify directly the convergence
\[
\M S(f) \to S\bigl(X^0,f\bigr) = \M\varGamma^{\top}
\bigl(X^0\bigr)V_A^{-1}ff^{\top
}V_A^{-1}
\varGamma\bigl(X^0\bigr) \quad\text{as } m\to\infty.
\]
Therefore, $S(f)\prob S(X^0,f)$.

(b) Without any problem, in view of Theorem \ref{th:2} and the
consistency of estimators $\hat{V}_A$ and $\hat{f}$, the following
convergences can be shown:
\begin{align*}
&S(f) - \hat{S}(f)\prob0, \quad \hat{S}(f):=\frac{1}{m}\sum
_{i=1}^{m}s^{\top}(a_i,b_i;
\hat{X})\cdot Z\cdot s(a_i,b_i;\hat{X});\\
&\hat{S}(f) - \hat{S}(\hat{f})\prob0.
\end{align*}
Here $Z$ is the matrix from relations \eqref{zz}.

The desired convergence follows from the convergences established in
parts (a) and (b) of the proof.\vadjust{\eject} \end{proof}
\begin{proof}[Proof of Lemma \ref{l:9}]
For model \eqref{MLM1}--\eqref{MLM2}, we have:
\begin{align}
\hspace*{24pt}\sqrt{m}T_m^0 &= \frac{1}{\sqrt{m}}\sum
_{i=1}^{m}\bigl(b_i^0 +
\tilde{b}_i - \hat{X}^{\top}a_i^0 -
\hat{X}^{\top}\tilde{a}_i^0\bigr)\\
&= \frac{1}{\sqrt
{m}}\sum_{i=1}^{m}\bigl(
\tilde{b}_i - X^{0\top}\tilde{a}_i\bigr) -
\sqrt{m}\bigl(\hat{X} - X^0\bigr)^{\top}\overline{a^0}
+ \mathit{rest},\\
 \mbox{where}\ \mathit{rest} &= -\bigl(\hat{X} - X^0\bigr)^\top\cdot
\frac{1}{\sqrt{m}}\sum_{i=1}^{m}\tilde
{a}_i = o_p(1)\cdot O_p(1) =
o_p(1).\notag\hspace*{41pt}\qedhere
\end{align}
\end{proof}

\begin{proof}[Proof of Lemma \ref{l:10}]
By Theorem \ref{th:4}\ref{4.b},
\[
\sqrt{m}\bigl(\hat{X} - X^0\bigr) = O_p(1).
\]
Therefore, expansion \eqref{STE} and condition \eqref{vii} imply that
\begin{equation}
\label{STE1} \sqrt{m}T_m^0 = \frac{1}{\sqrt{m}}\sum
_{i=1}^{m}\bigl(\tilde{b}_i -
X^{0\top
}\tilde{a}_i\bigr) + \sqrt{m}\bigl(\hat{X} -
X^0\bigr)^{\top}\mu_a + o_p(1).
\end{equation}
Next, by expansion \eqref{AEX} we get:
\begin{equation}
\label{STE2} \sqrt{m}T_m^0 = \frac{1}{\sqrt{m}}\sum
_{i=1}^{m}\bigl(\tilde{b}_i -
X^{0\top
}\tilde{a}_i + s^{\top}\bigl(a_i,b_i;X^0
\bigr)V_A^{-1}\mu_a\bigr) + o_p(1).
\end{equation}

The random vectors
\begin{equation}
\label{ZI} z_i := \tilde{b}_i - X^{0\top}
\tilde{a}_i + s^{\top
}\bigl(a_i,b_i;X^0
\bigr)V_A^{-1}\mu_a
\end{equation}
satisfy condition \eqref{BM} with the number $\delta$ from assumptions
\eqref{iv} and \eqref{v}. Let us find the variance--covariance matrix
$\varSigma_i$ of vector \eqref{ZI}. We have
\begin{equation}
\label{SI} \varSigma_i = \cov\bigl(\tilde{b}_i -
X^{0\top}\tilde{a}_i\bigr) + \cov\bigl(s^{\top
}
\bigl(a_i,b_i;X^0\bigr)~V_A^{-1}
\mu_a\bigr) + M + M^{\top}.
\end{equation}
Here (see \eqref{RM} and \eqref{EFE})
\begin{align}
M &:= \M s^{\top}\bigl(a_i,b_i;X^0
\bigr)~V_A^{-1}\mu_a\bigl(\tilde{b}_i
- \tilde{a}_i X^{0}\bigr)\notag\\
&=\M\bigl(X^{0\top}\tilde{a}_i a_i^{0\top}
- \tilde{b}_i a_i^{0\top
}\bigr)~V_A^{-1}
\mu_a\bigl(\tilde{b}_i^{\top} -
\tilde{a}_i^{\top} X^{0}\bigr);\notag\\
M &= -X^{0\top}\bigl(\M\tilde{a}_i a_i^{0\top}V_A^{-1}
\mu_a\tilde{a}_i^{\top
}\bigr)X^0 -
\M\tilde{b}_i a_i^{0\top}V_A^{-1}
\mu_a\tilde{b}_i^{\top}
\\
&= - a_i^{0\top}V_A^{-1}
\mu_a\sigma^2\bigl(\mathrm{I}_d +
X^{0\top}X^0\bigr) = M^{\top};\\
\cov\bigl(&\tilde{b}_i - X^{0\top}\tilde{a}_i\bigr) = \sigma^2\bigl(\mathrm{I}_d +
X^{0\top}X^0\bigr);\notag\\
 \cov\bigl(s^{\top}\bigl(a_i,b_i;&X^0\bigr)\,V_A^{-1}\mu_a\bigr)=\M s^{\top
}\bigl(a_i,b_i;X^0\bigr)\cdot Z \cdot s\bigl(a_i,b_i;X^0\bigr).\notag
\end{align}

Then
\begin{align*}
\varSigma_T &:= \lim_{m\to\infty}\frac{1}{m}(
\varSigma_1 +\cdots+ \varSigma_m)\\
&= \sigma^2\bigl(\mathrm{I}_d + X^{0\top}X^0
\bigr) + S\bigl(X^0,\mu_a\bigr) - 2\mu_a^{\top}V_A^{-1}
\mu_a\sigma^2\bigl(\mathrm{I}_d +
X^{0\top}X^0\bigr),
\end{align*}
and this coincides with the right-hand side of equality
\eqref{ACT}.\vadjust{\eject}

Finally, the desired convergence follows from expansion \eqref{STE2} by
Lemma \ref{l:20} and Slutsky's lemma.
\end{proof}

\begin{proof}[Proof of Lemma \ref{l:11}]
The convergence $\hat{\mu}_a\probone\mu_a$ is established by SLLN. The
convergence
\[
\hat{\varSigma}_T\prob\varSigma_T
\]
follows from Theorem \ref{th:8} (the role of $f$ and $\hat{f}$ is
played by $\mu_a$ and $\hat{\mu}_a$, respectively) and the consistency
of estimators $\hat{\sigma}^2$, $\hat{\mu}_a$, and
$\hat{V}_A$.
\end{proof}

\begin{proof}[Proof of Lemma \ref{l:13}]
(a) Hereafter, for symmetric matrices $A$ and $B$, notation $A\geq B$
$(A>B)$ means that the matrix $A-B$ is positive semidefinite (positive
definite).

Condition \eqref{vi} ensures the independence of the matrix components
$\varGamma_i$ in relation \eqref{CW}. Therefore,
\[
S\bigl(X^0,\mu_a\bigr) \geq\cov\bigl(
\bigl(X^{0\top}\tilde{a}_i - \hat{b}_i\bigr)~
\mu_a^{\top
}V_A^{-1}\mu_a
\bigr) = \sigma^2\bigl(\mu_a^{\top}V_A^{-1}
\mu_a\bigr)^2\bigl(\mathrm{I}_d +
X^{0\top}X^0\bigr).
\]

From equality \eqref{ACT} we have
\begin{equation}
\label{STIN} \varSigma_T\geq\sigma^2\bigl(1 -
\mu_a^{\top}V_A^{-1}\mu_a
\bigr)^2\cdot\mathrm{I}_d.
\end{equation}
By condition \eqref{viii}, $V_A>\mu_a\mu_a^{\top}$. In the case $\mu_a
= 0$, we get
$\varSigma_T\geq\sigma^2\mathrm{I}_d>0$, and in the case $\mu_a \neq0$,
we put $z = V_A^{-1}\mu_a$ and obtain:
\[
z^{\top}V_Az=\mu^{\top}_aV^{-1}_A
\mu_a>\bigl(\mu^{\top}_az\bigr)^2=
\bigl(\mu^{\top
}_aV^{-1}_A
\mu_a\bigr)^2;
\]
thus, $1 > \mu_a^{\top}V_A^{-1}\mu_a~$, and inequality \eqref{STIN}
implies $\varSigma_T > 0$.

Statement (b) follows from statement (a) and Lemma \ref{l:11}\eqref
{11.b}.
\end{proof}

\begin{proof}[Proof of Lemma \ref{l:15}]
(a) The local alternative \eqref{LAL} is corresponding to the
perturbation matrix
\[
G^0 := %
\begin{bmatrix}
g^{\top}(a_1^0)\\
\vdots\\
g^{\top}(a_m^0)
\end{bmatrix} %
.
\]
Model \eqref{LAL} can be rewritten as a perturbed model \eqref{OLSy},
\begin{equation}
\label{oLSyp} A^0 X^0 = B^0, \qquad A =
A^0 + \tilde{A}, \qquad B^{\mathit{per}} := B^0 +
\frac{1}{\sqrt{m}}G^0 + \tilde{B},
\end{equation}
or as a perturbed model \eqref{EM},
\begin{gather}
C^0 = \bigl[A^0 \ \ B^0\bigr], \qquad\tilde{C} =
[\tilde{A} \ \ \tilde{B}], \qquad C^{\mathit{per}} := \bigl[A \ \ B^{\mathit{per}}
\bigr], \label{EMP}
\\
C^0\cdot %
\begin{bmatrix}
X^0\\
-\mathrm{I}_d
\end{bmatrix} %
= 0. \label{EMO}
\end{gather}

Introduce the symmetric matrix
\[
N = C^{0\top}C^0 + \lambda_{\min}
\bigl(A^{0\top}A^0\bigr)\mathrm{I}_{n+d}.
\]
Due to condition \eqref{iii}, as $m\to\infty$,
\begin{equation}
\label{NAS} N = m N^0 + o(m), \quad N^0 > 0.\vadjust{\eject}
\end{equation}

Consider two matrices of size $(n+d)\times(n+d)$:
\begin{align}
M_1 &= N^{-1/2}C^{0\top}\bigl(C^{\mathit{per}} -
C^0\bigr)N^{-1/2},
\\
M_2 &= N^{-1/2}\bigl(\bigl(C^{\mathit{per}} - C^0
\bigr)^{\top}\bigl(C^{\mathit{per}} - C^0\bigr) -
\sigma^2 m \mathrm{I_{n+d}}\bigr)N^{-1/2}.
\end{align}
In view of the proof of Theorem 4.1 in \cite{skl}, for the convergence
\begin{equation}
\label{WCX} \hat{X} \prob X^0,
\end{equation}
it suffices to show that, as $m\to\infty$,
\begin{equation}
\label{CM} M_1\prob0, \qquad M_2\prob0,
\end{equation}
or taking into account \eqref{NAS}, that
\begin{align}
M_1^{\prime} &:= \frac{1}{m}C^{0\top}\tilde{C} +
\frac{1}{m}C^{0\top
}\cdot\frac{1}{\sqrt{m}}G^0
\prob0,\label{CM1}
\\
M_2^{\prime} &:= \frac{1}{m} \biggl( \biggl(\tilde{C} +
\frac{1}{\sqrt{m}}\bigl[0 \ \ G^0\bigr] \biggr)^{\top} \biggl(
\tilde{C} + \frac{1}{\sqrt{m}}\bigl[0 \ \ G^0\bigr] \biggr) -
\sigma^2\mathrm{I}_{n+d} \biggr) \prob0.\label{CM2}
\end{align}

We study the most interesting summands, those that contain $G^0$ (the
convergence of other summands was shown in the proof of Theorem 4.1 in
\cite{skl}). We have
\[
M_1^{\prime\prime} := \frac{1}{m^{3/2}}C^{0\top}G^0
= \frac{1}{m^{3/2}} %
\begin{bmatrix}
A^{0\top}G^0\\
X^{0\top} A^{0\top} G^0
\end{bmatrix} %
,
\]
and due to condition \eqref{ix}, as $m\to\infty$,
\begin{gather}
\frac{1}{m^{3/2}}A^{0\top}G^0 = \frac{1}{m^{3/2}}\sum
_{i=1}^{m}a_i^0
g^{\top}\bigl(a_i^0\bigr) = \frac
{O(1)}{m^{1/2}}
\to0,\quad M_1^{\prime\prime}\to0.
\end{gather}
Next, by condition \eqref{x},
\begin{gather}
M_2^{\prime\prime} := \frac{1}{m^2} G^{0\top}G^0
= \frac{1}{m^2}\sum_{i=1}^{m}g
\bigl(a_i^0\bigr)g^{\top}\bigl(a_i^0
\bigr),
\\
\big\|M_2^{\prime\prime}\big\| \leq\frac{\mathit{const}}{m^2}\sum
_{i=1}^{m}\big\|g\bigl(a_i^0
\bigr)\big\|^2\to0 \quad\text{as}~~ m\to\infty.
\end{gather}
Finally,
\begin{align}
M_2^{\prime\prime\prime} &:= \frac{1}{m^{3/2}}\tilde{C}^{\top}G^0
= \frac{1}{m^{3/2}}\sum_{i=1}^{m}
\tilde{c}_i g^{\top}\bigl(a_i^0
\bigr),
\\
\M\big\|M_2^{\prime\prime\prime}\big\|^2&\leq \frac{\mathit{const}}{m^{3}}\sum
_{i=1}^{m}\big\|g\bigl(a_i^0
\bigr)\big\|^2\to0 \quad\text{as} ~~m\to \infty,~~ M_2^{\prime\prime\prime}
\prob0.
\end{align}

We established the convergence in probability for the summands from
\eqref{CM1} and \eqref{CM2} that contain the perturbation $G^0$.
Therefore, \eqref{CM1} and \eqref{CM2} are satisfied, relation \eqref
{CM} is satisfied as well, and the results of \cite{skl} imply
convergence \eqref{WCX}.\vadjust{\eject}

The consistency of the estimator $\hat{\sigma}^2$ under local
alternatives $\textbf{H}_{1,m}$ is established by formula \eqref{ES}
and boils down to the consistency of $\hat{\sigma}^2$ under the null
hypothesis: the consistency of $\hat{X}$ has been proven already, and, moreover,
\begin{align}
\overline{b^{\mathit{per}}b^{\mathit{per},\top}} &= \frac{1}{m}\sum
_{i=1}^{m} \biggl(b_i +
\frac{1}{\sqrt{m}}g\bigl(a_i^0\bigr) \biggr)
\biggl(b_i + \frac{1}{\sqrt{m}}g\bigl(a_i^0
\bigr) \biggr)^{\top} = \frac{1}{m}\sum
_{i=1}^{m}b_i b_i^{\top}
+ o_p(1),
\\
\overline{ab^{\mathit{per},\top}}& = \overline{a b^{\top}} +
o_p(1).
\end{align}

(b) After we established the consistency of $\hat{X}$ under
alternatives $\textbf{H}_{1,m}$, we find an expansion similar to \eqref{AEX}:
\begin{align}
\sqrt{m}\bigl(\hat{X} - X^0\bigr) &= -V_A^{-1}
\frac{1}{\sqrt{m}}\sum_{i=1}^{m}s
\bigl(a_i,b_i^{\mathit{per}}; X^0\bigr) +
o_p(1)
\\
&= -V_A^{-1}\frac{1}{\sqrt{m}}\sum
_{i=1}^{m}s\bigl(a_i,b_i;
X^0\bigr) - V_A^{-1}\frac{1}{\sqrt{m}}\sum
_{i=1}^{m}s_i^{\mathit{per}} +
o_p(1).\label{AEX1}
\end{align}
Conditions \eqref{ix} and \eqref{x} ensure that, for perturbations
$s_i^{\mathit{per}}$ of the estimating function, we have (the main contribution
to $s_i^{\mathit{per}}$ is made by a linear summand $ab^{\top}$ from \eqref{EF}):
\begin{equation}
\label{PE} \frac{1}{\sqrt{m}}\sum_{i=1}^{m}s_i^{\mathit{per}}
= - \frac{1}{m}\sum_{i=1}^{m}
a_i^0 g^{\top}\bigl(a_i^0
\bigr) + o_p(1).
\end{equation}

Lemma \ref{l:7}, Theorem \ref{th:4}, and formulae \eqref{AEX1} and
\eqref{PE} imply the desired convergence of the normalized TLS
estimator.
\end{proof}

\begin{proof}[Proof of Lemma \ref{l:16}]
(a) Under the local alternatives, we have:
\[
\sqrt{m}T_m^0 |_{\textbf{H}_{1,m}} = \sqrt{m}T_m^0
|_{\textbf
{H}_{0}} + M\bigl(g\bigl(a^0\bigr)\bigr) - \sqrt{m} (\hat{X}
|_{\textbf{H}_{1,m}} - \hat{X} |_{\textbf{H}_{0}} )^{\top}\mu_a +
o_p(1).
\]
Expansions \eqref{AEX}, \eqref{AEX1}, and \eqref{PE} imply that
\[
\sqrt{m} (\hat{X} |_{\textbf{H}_{1,m}} - \hat{X} |_{\textbf
{H}_{0}} ) \prob
V_A^{-1}\cdot M\bigl(a^0 g^{\top}
\bigl(a^0\bigr)\bigr).
\]
Then, by Lemma \ref{l:10} and Slutsky's lemma,
\begin{align}
&\sqrt{m}T_m^0 |_{\textbf{H}_{1,m}}\convdistr
N(C_T,\varSigma_T),
\\
&\quad C_T = M\bigl(g\bigl(a^0\bigr)\bigr) - M\bigl(g
\bigl(a^0\bigr)a^{0\top}\bigr)\cdot V_A^{-1}
\mu_a~.
\end{align}

(b) Under the local alternatives, the estimators $\hat{\sigma}^2$, $\hat
{\mu}_a$, $\hat{V}_A$, and $\hat{X}$ are still consistent. Moreover,
\begin{equation}
\label{EACL} \hat{S}(\hat{\mu}_a) = \frac{1}{m}\sum
_{i=1}^{m}s^{\top}\bigl(a_i,b_i^{\mathit{per}};
\hat{X}\bigr)~\hat{V}_A^{-1}\hat{\mu}_a \hat{
\mu}_a^{\top}\hat {V}_A^{-1}~s
\bigl(a_i,b_i^{\mathit{per}}; \hat{X}\bigr)
\end{equation}
converges in probability to $S(X^0,\mu_a)$ because expression \eqref
{EACL} does not involve terms linear in $b_i^{\mathit{per}}$, and the
perturbation of the vectors $b_i$ does not modify the asymptotic
behavior of $\hat{S}(\hat{\mu}_a)$ in transition from $\textbf{H}_{0}$
to the local alternatives.

Thus, estimator \eqref{EACS} does converge in probability to matrix
\eqref{ACT}.
\end{proof}

\end{document}